\documentclass{article}

\usepackage[
  left=1.25in,
  right=1.25in,
  top=1.2in,
  bottom=1.2in
]{geometry}

\usepackage{amsmath,amssymb,amsthm}
\usepackage{mathtools}
\usepackage{hyperref}

\usepackage[
  style=numeric-comp,
  hyperref=true,
  doi=false,
  url=false,
  isbn=false,
  giveninits=true,
  sorting=none,
  block=none,
  backend=biber,
  maxnames=99
]{biblatex}

\newcommand{\C}{\mathbb{C}}
\newcommand{\tr}{\operatorname{tr}}
\newcommand{\frakg}{\mathfrak{g}}
\newcommand{\frakh}{\mathfrak{h}}
\newcommand{\xalg}{\mathfrak X_{\mathrm{alg}}}

\newtheorem{theorem}{Theorem}[section]
\newtheorem{proposition}[theorem]{Proposition}
\newtheorem{lemma}[theorem]{Lemma}
\newtheorem{remark}[theorem]{Remark}

\hypersetup{
  hidelinks,
  pdftitle={Integrable Generators of Polynomial Vector Fields for Complex Classical Lie Groups},
  pdfauthor={Yiyang Jiang and Xudong Chen}
}

\title{Integrable Generators of Polynomial Vector Fields for\\
Complex Classical Lie Groups}

\author{Yiyang Jiang$^*$ \quad and \quad
Xudong Chen\footnote{Y. Jiang and X. Chen are with the Department of Electrical \& Systems Engineering, Washington University, St. Louis, MO 63130, USA.
Emails: {\small\texttt{j.yiyang@wustl.edu, cxudong@wustl.edu}}. 
Corresponding author: Y. Jiang.}}

\date{}

\begin{document}

\maketitle

\begin{abstract}
We consider in this paper complex Lie groups $\mathrm{SL}_n(\C)$ for $n\geq 2$, $\mathrm{SO}_n(\C)$ for $n\geq 5$, and $\mathrm{Sp}_{2n}(\C)$ for $n\geq 1$.
We show that for any such Lie group, the Lie algebra of polynomial vector fields can be generated by three $\C$-complete polynomial vector fields.
\end{abstract}

\medskip

\section{Introduction and main result}
\label{sec:introduction}

Let $\mathrm{GL}_d(\C)$ be the general linear group of all invertible $d$-by-$d$ matrices over $\C$. 
Let $G\subseteq \mathrm{GL}_d(\C)$ be one of the following complex Lie groups: 
\begin{equation}\label{eq:listofgroups}
\begin{aligned}
\mathrm{SL}_n(\C) & := \{x\in \mathrm{GL}_n(\C) \mid \det(x) = 1\}, \quad \mbox{for } n \geq 2,\\
\mathrm{SO}_n(\C) & := \{x\in \mathrm{SL}_n(\C) \mid x^\top x = I \}, \quad \mbox{for } n \geq 5,\\
\mathrm{Sp}_{2n}(\C) & := \{x\in \mathrm{SL}_{2n}(\C) \mid x^\top \Omega x = \Omega\}, \quad \mbox{for } n \geq 1,
\end{aligned}
\end{equation} 
where $\Omega := [0, I_n; -I_n,0]$. 
We also list these groups in Table~\ref{tab:groups}, together with their types of the associated root systems.

A vector field $V$ on $G$ is viewed as a map from $G$ to $\C^{d\times d}$.  
The vector field $V$ is said to be {\it $\C$-complete} if its flow $e^{tV}: G\to G$ is well defined for every $t\in \C$. 
We use $\xalg(G)$ to denote the Lie algebra of polynomial vector fields on~$G$. Given a set $S$ of holomorphic vector fields, let $\operatorname{Lie}(S)$ be the Lie algebra generated by~$S$. In this paper, all vector spaces, Lie algebras, etc., are over~$\C$, unless otherwise stated. 

The main result of the paper is the following:

\begin{theorem}\label{thm:main}
For any complex Lie group $G$ in~\eqref{eq:listofgroups}, there exist three $\C$-complete polynomial vector fields $W_1$, $W_2$, and $W_3$ such that $\xalg(G)=\operatorname{Lie}\{W_1,W_2,W_3\}$.
\end{theorem}

We will exhibit explicitly the three vector fields in the proof of the theorem, with the first two given in Section~\ref{sec:firsttwo} and third in Section~\ref{sec:third}. 

It has been shown in~\cite{TV} that every complex semi-simple Lie group~$G$ has the so-called {\it density property}, namely, the property that every holomorphic vector field $V$ on $G$ can be approximated arbitrarily well, with respect to the compact-open topology, by Lie combinations of all complete vector fields. 
Our result then implies that any such $V$ can be approximated by Lie combinations of a fixed, finite set of complete vector fields. We call this property the {\it strong density property}. 
As a consequence, the group of diffeomorphisms generated by the flows $e^{t W_i}$, for $i = 1,2,3$ and for $t\in \C$, is dense, with respect to the compact-open topology, in the identity component of the group of holomorphic diffeomorphisms.  

The strong density property for complex Lie groups has been investigated in~\cite{AndristSL} for the special case $\mathrm{SL}_2(\C)$, where the author has shown that $\xalg(\mathrm{SL}_2(\C))$ can be generated by four complete polynomial vector fields. In the same paper, the author has posed the following question: ``Can we find finitely many complete polynomial vector fields that generate the Lie algebra of all polynomial vector fields on other simple Lie groups than $\mathrm{SL}_2(\C)$?'' Our result thus provides an affirmative answer to this question. 

The strong density property has also been investigated in other complex affine varieties, including~$\C^n$ for $n\geq 2$~\cite{Andrist,BeldievNew} and Danielewski surfaces~\cite{AndristSL,And26}.   
The same property, but for the class of real, connected, semi-simple matrix Lie groups, has been addressed in~\cite{Real}, in which we have shown that the Lie algebra of algebraic vector fields over~$\mathbb{R}$ can be generated by three $\mathbb{R}$-complete vector fields.    

\section{The first two  vector fields}
\label{sec:firsttwo}

Let $\frakg$ be the Lie algebra of $G$. We use $\operatorname{ad}_X$ to denote the adjoint action, i.e., $\operatorname{ad}_X(Y) = [X, Y]$, for $X, Y\in \frakg$.  
Let $\frakg_L$ and $\frakg_R$ be the space of left- and right-invariant vector fields on $G$, respectively. We identify $\frakg$ with $\frakg_L$. For a given $X\in \frakg$, let $L_X(x):=xX$ and $R_X(x):= X x$ be the corresponding left- and right-invariant vector field, respectively. We have the following commutator relations: 
$$
[L_X,L_Y]=L_{[X,Y]},\qquad
[R_X,R_Y]=-R_{[X,Y]},\qquad
[L_X,R_Y]=0.
$$
It is clear that every left- or right-invariant vector field is algebraic. Also, for any $X,Y\in\frakg$, 
\begin{equation}\label{eq:invariant-flow}
e^{t(L_X+R_Y)}(x)=e^{tY}xe^{tX},
\end{equation}
for any $x\in G$ and any $t\in\C$. It follows that every vector field in $\frakg_L\oplus\frakg_R$ is $\C$-complete. 

\begin{table}[th]
\centering
\begin{tabular}{c|c}
$G$ & Type\\
\hline
$\mathrm{SL}_n(\C)$
& $A_{n-1}$, $n\geq2$\\

$\mathrm{SO}_{2r+1}(\C)$
& $B_r$, $r\geq2$\\

$\mathrm{Sp}_{2r}(\C)$
& $C_r$, $r\geq1$\\

$\mathrm{SO}_{2r}(\C)$
& $D_r$, $r\geq3$\\
\hline
\end{tabular}
\caption{Classical complex groups $G\subseteq \mathrm{GL}_d(\C)$ considered in the paper, together with the types of their root systems. The subindex of each type indicates the rank of the associated Lie algebra $\frakg$.}
\label{tab:groups}
\end{table}

We will now exhibit for each Lie group $G$ in Table~\ref{tab:groups} two vector fields $W_1$ and $W_2$ such that $\frakg_L \oplus \frakg_R = \operatorname{Lie}\{W_1, W_2\}$.   
In all cases, $W_1$ and $W_2$ take the following form:
\begin{equation}\label{eq:firsttwo}
W_1:=L_H-R_{H'}
\quad\text{and}\quad
W_2:=L_P-R_P,
\end{equation}
for some $H, H', P\in \frakg$. 
We present these three matrices for each type of~$\frakg$. Throughout the presentation, we use $E_{ij}\in \C^{d\times d}$ to denote the elementary matrix with~$1$ on the $ij$th entry and $0$ elsewhere. Further, let $S_{ij} := E_{ij} + E_{ji}$ and $\Omega_{ij}:= E_{ij} - E_{ji}$.    

\begin{description}
    \item[\it Type $A_{n-1}$ for $n \geq 2$:] The Lie algebra $\mathfrak{sl}_{n}(\C)$ is given by
    $$
    \mathfrak{sl}_n(\C)= \{ X\in \C^{n\times n} \mid \tr(X) = 0\}.
    $$
    We define 
    \begin{equation}\label{eq:typeAmatrices}
    \begin{aligned}
    H & := \sum_{i = 1}^n h_i E_{ii}, \quad \mbox{with } h_i:=\frac{1}{n}(2^{n+1}-2) -2^i, \\
    H'& :=  2^{n-1} H, \\
    P &:= \sum_{i = 1}^{n-1} S_{i,i+1}.
    \end{aligned}
    \end{equation}
    
    \item[\it Type $B_r$ for $r\geq 2$:] The Lie algebra $\mathfrak{so}_{d}(\C)$, with $d = 2r + 1$, is given by
    $$
    \mathfrak{so}_{d}(\C) = \{X\in \C^{d\times d} \mid X^\top + X = 0\}.
    $$
    We define 
    \begin{equation}\label{eq:typeBmatrices}
    \begin{aligned}
    H & := \sum_{i = 1}^r h_i \Omega_{i,d+1-i},\quad \mbox{with } h_i:= 2^{r+1} -2^i, \\
    H'& :=  2^r H, \\
    P &:= \sum_{i = 1}^{r} (\Omega_{i,i+1} - \Omega_{r+i,r+i+1}).
    \end{aligned}
    \end{equation}

\item[\it Type $C_r$ for $r \geq 1$:] 
The Lie algebra $\mathfrak{sp}_{2r}(\C)$ is given by
$$
\mathfrak{sp}_{2r}(\C)=\{X\in\C^{2r\times2r}\mid X^\top\Omega+\Omega X=0\},
$$
where we recall $\Omega= \sum_{i = 1}^r \Omega_{i, i + r}$.
We define
\begin{equation}\label{eq:typeCmatrices}
\begin{aligned}
H&:= \sum_{i = 1}^r h_i (E_{ii} - E_{i+r,i+r}), 
\quad \mbox{with } h_i:=3\cdot2^{r-1}-2^i,\\
H'&:=2^rH,\\
P&:=S_{r,2r} + \sum_{i=1}^{r-1}(S_{i,i+1}-S_{r+i+1,r+i}).
\end{aligned}
\end{equation}

\item[\it Type $D_r$ for $r\geq3$:]
The Lie algebra $\mathfrak{so}_{d}(\C)$, with $d=2r$, is given by
$$
\mathfrak{so}_{d}(\C)=\{X\in\C^{d\times d}\mid X^\top+X=0\}.
$$
We define
\begin{equation}\label{eq:typeDmatrices}
\begin{aligned}
H&:=\sum_{i = 1}^r h_i \Omega_{i,d+1-i},\quad \mbox{with } h_i:= 5\cdot 2^{r-2} -2^i, \\
H'&:=2^rH,\\
P&:=(\Omega_{r-1,r+1}-\Omega_{r,r+2}) + \sum_{i=1}^{r-1}(\Omega_{i,i+1}-\Omega_{2r-i,2r+1-i}).
\end{aligned}
\end{equation}
\end{description}

It is clear that in all cases the matrices $H$, $H'$, and $P$ belong to their respective Lie algebra, so the two vector fields $W_1$ and $W_2$ are well defined on~$G$.  
We now have the following proposition: 

\begin{proposition}\label{prop:firsttwo}
It holds that $\operatorname{Lie}\{W_1,W_2\}=\frakg_L\oplus\frakg_R$.
\end{proposition}

\begin{proof}
Let $\frakh$ be a Cartan subalgebra of $\frakg$, and $\Delta$ be the root system associated with $(\frakg, \frakh)$. Choose a positive system $\Delta^+$ and let $\alpha_1,\ldots, \alpha_r$ be the simple roots, where $r$ is the rank of $\frakg$. Similar to the arguments of~\cite{Real,Kuranishi}, we have that if $H$, $H'$, and $P$ satisfy the following two items:
\begin{enumerate}
\item Matrices $H$ and $H'$ belong to $\frakh$. The $4r$ numbers  $\pm\alpha_i(H), \pm\alpha_i(H')$, for $1\leq i \leq r$,  are pairwise distinct;
\item There exist root vectors $X_{\alpha_i}$ and $X_{-\alpha_i}$, for $1\leq i\leq r$, such that $P = \sum_{i = 1}^r (X_{\alpha_i} + X_{-\alpha_i})$; 
\end{enumerate}
then one obtains by computation that
\begin{multline*}
\operatorname{ad}_{W_1}^m W_2 = \sum_{i = 1}^r \left (\alpha_i(H)^m L_{X_{\alpha_i}} + (-\alpha_i(H))^m L_{X_{-\alpha_i}} \right. \left. -  \alpha_i(H')^m R_{X_{\alpha_i}} - (-\alpha_i(H'))^m R_{X_{-\alpha_i}} \right ), 
\end{multline*}
for any $m \geq 0$ and hence, 
$$\{L_{X_{\alpha_i}}, L_{X_{-\alpha_i}}, R_{X_{\alpha_i}}, R_{X_{-\alpha_i}}\mid 1\leq i\leq r\}\subseteq \operatorname{Lie}\{W_1,W_2\}.$$
Since the set on the left hand side of the above equation generates the Lie algebra $\frakg_L\oplus \frakg_R$ (see, e.g.,~\cite[Section~14.2]{Humphreys}), we then have that $\operatorname{Lie}\{W_1,W_2\} = \frakg_L \oplus \frakg_R$. Thus, to establish the proposition, it suffices to verify that the matrices $H$, $H'$, and $P$ satisfy the two items above. We do this below.

\begin{description}
    \item[\it Validation for type $A_{n-1}$:] Let $\frakh$ be the space of trace-less diagonal matrices. Then, $\Delta = \{ e_i - e_j \mid 1\leq i\neq j \leq n\}$, where $e_i(A):= A_{ii}$ for any $A\in \frakh$. We choose the positive system such that $e_i - e_j\in \Delta^+$ if $i < j$. Then, the set of simple roots is $\{e_{i} - e_{i + 1} \mid 1\leq i \leq n-1 \}$. By  definition~\eqref{eq:typeAmatrices}, $\alpha_i(H) = 2^i$ and $\alpha_i(H') = 2^{n + i - 1}$, 
    for $i = 1,\ldots, n-1$, so item~1 is satisfied. Also, note that $E_{i,i+1}$ is a root vector corresponding to $(e_i - e_{i+1})$, so $P$ satisfies item~2. 
    \item[\it Validation for type $B_{r}$:] Let $\frakh:=\operatorname{span}\{\Omega_{i,d+1-i}\mid 1\leq i\leq r\}$. 
    For an arbitrary $A=\sum_{i=1}^r a_i\Omega_{i,d+1-i}$ in~$\frakh$, we define $e_i(A):=\mathrm{i}a_i$, where $\mathrm{i}$ is the imaginary unit. Then, the root system is given by $\Delta=\{\pm e_i\pm e_j\mid 1\leq i<j\leq r\}\cup \{\pm e_i\mid 1\leq i\leq r\}$.  
    We choose the positive system such that $e_i-e_j,e_i+e_j\in\Delta^+$ if $i<j$ and $e_i\in\Delta^+$ for any~$i$. Then, the simple roots are $\alpha_i:=e_i-e_{i+1}$, for $1\leq i\leq r-1$, and $\alpha_r:=e_r$.  By~\ref{eq:typeBmatrices},
    $\alpha_i(H)=\mathrm{i}2^i$ and $\alpha_i(H')=\mathrm{i}2^{r+i}$, 
    for $i=1,\ldots,r$, so item~1 is satisfied. 
    Next, for $1\leq i \leq r$, let
    $$
    \begin{aligned}
    X_{\alpha_i}&:=\frac{1}{2}(\Omega_{i,i+1}-\Omega_{d-i,d+1-i})-\frac{\mathrm{i}}{2}(\Omega_{i,d-i}+\Omega_{i+1,d+1-i}), \\
    X_{-\alpha_i}&:=\frac{1}{2}(\Omega_{i,i+1}-\Omega_{d-i,d+1-i}) +\frac{\mathrm{i}}{2}(\Omega_{i,d-i}+\Omega_{i+1,d+1-i}). 
    \end{aligned}
    $$
    It follows directly from computation that $X_{\alpha_i}$ and $X_{-\alpha_i}$ are root vectors corresponding to $\alpha_i$ and $-\alpha_i$, respectively. Moreover, the matrix $P$ given in~\eqref{eq:typeBmatrices} satisfies $P = \sum_{i = 1}^r (X_{\alpha_i} + X_{-\alpha_i})$. 
    \item[\it Validation for type $C_{r}$:] Let $\frakh$ be the space of matrices $A=\sum_{i=1}^r a_i(E_{ii}-E_{i+r,i+r})$, and let $e_i(A):=a_i$. 
    Then, $\Delta=\{\pm e_i\pm e_j\mid 1\leq i<j\leq r\}\cup\{\pm2e_i\mid 1\leq i\leq r\}$. We choose the positive system such that $e_i-e_j,e_i+e_j\in\Delta^+$ if $i<j$ and $2e_i\in\Delta^+$ for any~$i$. The simple roots are $\alpha_i := e_i-e_{i+1}$ for $1\leq i\leq r-1$ and $\alpha_r:=2e_r$. By~\ref{eq:typeCmatrices},
    $\alpha_i(H)=2^i$ and $\alpha_i(H')=2^{r+i}$, 
    for $i=1,\ldots,r$, so item~1 is satisfied. Let
    $$
    \begin{aligned}
    X_{\alpha_i}&:=E_{i,i+1}-E_{r+i+1,r+i},&
    X_{-\alpha_i}&:=E_{i+1,i}-E_{r+i,r+i+1},
    \quad \mbox{for } 1\leq i<r,\\
    X_{\alpha_r}&:=E_{r,2r},&
    X_{-\alpha_r}&:=E_{2r,r}.
    \end{aligned}
    $$
    By computation, $X_{\alpha_i}$ and $X_{-\alpha_i}$ are root vectors corresponding to $\alpha_i$ and $-\alpha_i$, respectively. Moreover, the matrix $P$ given in~\eqref{eq:typeCmatrices} satisfies $P=\sum_{i=1}^r(X_{\alpha_i}+X_{-\alpha_i})$.
    \item[\it Validation for type $D_{r}$:] 
    Similar to type $B_r$, we let $\frakh:=\operatorname{span}\{\Omega_{i,d+1-i}\mid 1\leq i\leq r\}$. For $A=\sum_{i=1}^r a_i\Omega_{i,d+1-i}\in\frakh$, let $e_i(A):=\mathrm{i}a_i$. Then, $\Delta=\{\pm e_i\pm e_j\mid 1\leq i<j\leq r\}$. We choose the positive system such that $e_i-e_j,e_i+e_j\in\Delta^+$ if $i<j$. The simple roots are $\alpha_i:=e_i-e_{i+1}$, for $1\leq i \leq r - 1$, and $\alpha_r:=e_{r-1}+e_r$. 
    By~\ref{eq:typeDmatrices},
    $\alpha_i(H)=\mathrm{i}2^i$ and $\alpha_i(H')=\mathrm{i}2^{r+i}$, 
    for $i=1,\ldots,r$, so item~1 is satisfied. 
    Next, for $1\leq i<r$, let
    $$
    \begin{aligned}
    X_{\alpha_i}&:=\frac{1}{2}(\Omega_{i,i+1}-\Omega_{d-i,d+1-i})-\frac{\mathrm{i}}{2}(\Omega_{i,d-i}+\Omega_{i+1,d+1-i}), \\
    X_{-\alpha_i}&:=\frac{1}{2}(\Omega_{i,i+1}-\Omega_{d-i,d+1-i}) +\frac{\mathrm{i}}{2}(\Omega_{i,d-i}+\Omega_{i+1,d+1-i}),
    \end{aligned}
    $$
    and let
    $$
    \begin{aligned}
    X_{\alpha_r}& :=\frac{1}{2}(\Omega_{r-1,r+1}-\Omega_{r,r+2}) - \frac{\mathrm{i}}{2}(\Omega_{r-1,r}+\Omega_{r+1,r+2}), \\
    X_{-\alpha_r} &:=\frac{1}{2}(\Omega_{r-1,r+1}-\Omega_{r,r+2}) + \frac{\mathrm{i}}{2}(\Omega_{r-1,r}+\Omega_{r+1,r+2}). 
    \end{aligned}
    $$
    It follows from computation that $X_{\alpha_i}$ and $X_{-\alpha_i}$ are root vectors corresponding to $\alpha_i$ and $-\alpha_i$, respectively, and that the matrix $P$ given in~\eqref{eq:typeDmatrices} satisfies $P=\sum_{i=1}^r(X_{\alpha_i}+X_{-\alpha_i})$. 
\end{description}
This completes the proof.
\end{proof}

\section{The third vector field}\label{sec:third}
In this section, we construct the third vector field~$W_3$ and prove Theorem~\ref{thm:main}. We start with the following lemma:

\begin{lemma}\label{lem:matrixQ}
    There exists a nonzero~$Q\in \frakg$ such that $Q^2 = 0$.  
\end{lemma}

\begin{proof}
We exhibit the desired matrix $Q$ for each type of~$\frakg$: 
\begin{description}
    \item[\it Type $A_{n-1}$ for $n \geq 2$:] Let $Q:= E_{12}$. It is clear that $Q^2 = 0$.
    \item[\it Type $B_{r}$ for $r \geq 2$ and type $D_r$ for $r \geq 3$:] For either case, let 
    \begin{equation}\label{eq:defQ}
    Q:= X_{\alpha_1} = \frac{1}{2}(\Omega_{1,2}-\Omega_{d-1,d})-\frac{\mathrm{i}}{2}(\Omega_{1,d-1}+\Omega_{2,d}) = : \frac{1}{2}U -\frac{\mathrm{i}}{2} V.
    \end{equation}
    By computation, $U^2 = V^2$ and $UV + VU = 0$, so $Q^2 = 0$.
    \item[\it Type $C_r$ for $r \geq 1$:] Let 
    $$
    Q:= 
    \begin{cases}
    E_{12} & \mbox{if } r = 1, \\
    E_{12} - E_{r + 2, r+1} & \mbox{if } r > 1.
    \end{cases}
    $$ 
    We have that $Q^2 = 0$. 
\end{description}
This completes the proof. 
\end{proof}

\begin{remark}
    The Lie algebra $\mathfrak{so}_3(\C)$ does not have any nonzero element $Q$ such that $Q^2 = 0$. 
\end{remark}

Let $\chi:G\to \C$ be the character of the standard representation, i.e., $\chi(x):=\tr(x)$. With the nonzero $Q\in \frakg$ given in Lemma~\ref{lem:matrixQ}, we define 
\begin{equation}\label{eq:third}
W_3:=\chi L_Q.
\end{equation}
It is clear that $W_3\in\xalg(G)$. We have

\begin{lemma}
The vector field $W_3$ is $\C$-complete.
\end{lemma}

\begin{proof}
It is known~\cite[Proposition~2.3]{TV06} that a vector field $\phi V$ on a complex manifold, with $V$ a $\C$-complete holomorphic vector field and $\phi$ a holomorphic function, is $\C$-complete if and only if it is an {\it overshear}, i.e., $V^2\phi = 0$. 
In our case, $L_Q$ is $\C$-complete and holomorphic and $\chi$ is holomorphic, so it suffices to show that $L_Q^2\chi = 0$. But this directly follows from Lemma~\ref{lem:matrixQ}; indeed, $(L_Q^2\chi)(x) = \tr(xQ^2) = 0$. In fact, one can express the flow $e^{tW_3}$ explicitly as 
$$
e^{tW_3}(x) = x(I + \tau(t,x)Q),
$$
where 
$$
  \tau(t,x)=
  \begin{cases}
    \displaystyle \frac{e^{t\tr(xQ)}-1}{\tr(xQ)} \tr(x),& \mbox{if }\tr(xQ)\neq 0,\\
    t\tr(x),& \mbox{if } \tr(xQ) = 0.
  \end{cases}
$$
The above holds for all $t\in \C$ and for all $x\in G$. 
\end{proof}

Let $x_{ij}$ be the $ij$th entry of $x\in G$, viewed as a holomorphic function on~$G$. 
Let $\Phi$ be the space spanned by $x_{ij}$, for $ 1\leq i, j\leq d$. It is clear that $\Phi$ is closed under $\frakg_L \oplus\frakg_R$. 
Let $\Phi\frakg_L\subseteq \xalg(G)$  
be the space of vector fields spanned by $\phi L_X$ for $\phi\in\Phi$ and $X\in\frakg$. 
The next result is key to establishing Theorem~\ref{thm:main}.   Its proof is similar to that of~\cite[Theorem~1]{Real}. In fact, irreducibility of the standard representation and simplicity 
of $\frakg$ make the arguments simpler. For completeness of the presentation, we include a relatively short proof of the lemma: 

\begin{lemma}\label{lem:degreeone}
It holds that $\Phi\frakg_L\subseteq\operatorname{Lie}\{W_1,W_2,W_3\}$.
\end{lemma}

\begin{proof}
Let $S$ be the subspace of $\xalg(G)$ spanned by  $\operatorname{ad}_{V_m}\cdots \operatorname{ad}_{V_1} W_3$, 
for $V_i\in\frakg_L\oplus\frakg_R$ and for $m\geq 0$ (by default, the Lie product is simply $W_3$ for $m = 0$).  
By Proposition~\ref{prop:firsttwo}, it suffices to show that $\Phi\frakg_L\subseteq S$. First, using the fact that the left-invariant vector fields commute with the right-invariant ones, we obtain that 
\begin{equation}\label{eq:rightonW3}
    \operatorname{ad}_{R_{X_m}}\cdots\operatorname{ad}_{R_{X_1}}W_3
    =(R_{X_m}\cdots R_{X_1}\chi)L_Q.
\end{equation}
Note that 
\begin{equation}\label{eq:rightonchi}
(R_{X_m}\cdots R_{X_1}\chi)(x) = \tr(x X_1\cdots X_m).
\end{equation}
Let $\mathcal{A}(\frakg)$ be the unital algebra generated by the matrices of $\frakg$. 
Since the standard representation $\C^d$ of $\frakg$ is irreducible, Burnside's theorem~\cite{Burnside} implies that $\mathcal{A}(\frakg)$ is the space of all $d$-by-$d$ complex matrices. Combining the above arguments with~\eqref{eq:rightonW3} and~\eqref{eq:rightonchi}, we obtain that $\Phi L_Q\subseteq S$. 
Now, consider the set $\mathfrak{i}:=\{X\in\frakg \mid \Phi L_X\subseteq S\}$.  
It is clear from the definition that $\mathfrak{i}$ is a linear subspace of~$\frakg$. Also, we have shown that $Q\in \mathfrak{i}$, so $\mathfrak{i}$ is nontrivial. 
For any $X\in\mathfrak{i}$, any $Y\in\frakg$, and any $\phi\in\Phi$, 
$$[L_Y,\phi L_X]=(L_Y\phi)L_X+\phi L_{[Y,X]}.$$
Since $S$ is closed under~$\frakg_L$, the left hand side belongs to~$S$. Also, $L_Y\phi\in\Phi$ and $X\in\mathfrak{i}$, so $(L_Y\phi)L_X\in S$. It follows that $\phi L_{[Y,X]}\in S$ for every $\phi\in\Phi$.
Thus, $[Y,X]\in\mathfrak{i}$, and $\mathfrak{i}$ is a nonzero ideal of the simple Lie algebra~$\frakg$. Therefore, $\mathfrak{i}=\frakg$, and $\Phi\frakg_L\subseteq S\subseteq\operatorname{Lie}\{W_1,W_2,W_3\}$.
\end{proof}

\begin{remark}
For the Lie group $\mathrm{SO}_4(\C)$, one can still define the matrix $Q$ as in~\eqref{eq:defQ} and hence, the vector field $W_3$, which is $\C$-complete. It is, however, the case where $Q$ belongs to a nonzero simple ideal of $\mathfrak{so}_4(\C)$. Consequently, the arguments at the end of the proof of Lemma~\ref{lem:degreeone} no longer hold.    
\end{remark}

For each integer $n\geq 2$, let $\mathcal{A}_n(\Phi)$ be the space spanned by the products $\phi_1\cdots\phi_n$, with $\phi_i\in\Phi$ and $\mathcal{L}_n$ be the space of right-normed Lie products $\operatorname{ad}_{V_n} \cdots \operatorname{ad}_{V_2} V_1 $ of depth~$n$, with $V_i\in \Phi \frakg_L$. We have the following lemma:

\begin{lemma}\label{lem:degreegrow}
For any $n\geq 2$, $\mathcal{L}_n = \mathcal{A}_n(\Phi)\frakg_L$. 
\end{lemma}

\begin{proof}
This proof builds directly upon~\cite[Theorem~2]{Real} and its proof, which states that if $\Phi$ does not contain any nonzero constant function, then $\mathcal{L}_n = \mathcal{A}_n(\Phi)\frakg_L$.
(Note that the statement of that theorem is for the real setting, yet the proof works by complexification.)   

We show below that for our case, $\Phi$ indeed has no nonzero constant function. 
First, note that each function in~$\Phi$ takes the form $\phi_A(x):=\tr(xA)$ for some matrix $A\in \C^{d\times d}$. 
Suppose that $\phi_A$ is constant; then, for any $X\in\frakg$, $R_X\phi_A=\phi_{AX}=0$. 
Since $\C^d$ is irreducible under~$G$, we use again Burnside's theorem~\cite{Burnside} to have that the unital algebra $\mathcal{A}(G)$ generated by elements of~$G$ is $\C^{d\times d}$. 
Since $G$ is itself a group, $\mathcal{A}(G)$ is simply its linear span. 
It then follows that $\operatorname{span}G=\C^{d\times d}$. Thus, 
$$
\phi_{AX}(x) = \tr(xAX) = 0 \quad \mbox{for all } x\in G \quad \Rightarrow \quad AX = 0,
$$
which holds for any $X\in\frakg$. Now, let $U:=\operatorname{span}\{Xv \mid X\in\frakg,\ v\in\C^d\}$. 
The space $U$ is nonzero because the action $(X, v)\mapsto Xv$ is nontrivial. Moreover, for any $X,Y\in\frakg$ and any $v\in\C^d$, $Y(Xv)=[Y,X]v+X(Yv)\in U$, so $U$ is $\frakg$-invariant. 
Thus, by irreducibility, $U=\C^d$. Since $A$ vanishes on~$U$, we conclude that $A=0$. Therefore, $\Phi$ does not contain any nonzero constant function.
\end{proof}

With the results above, we now prove Theorem~\ref{thm:main}:

\begin{proof}[Proof of Theorem~\ref{thm:main}]
By construction, the three vector fields $W_1$, $W_2$, and $W_3$ are polynomial and $\C$-complete. 
Let $\mathcal{A}(\Phi)$ be the unital algebra generated by elements of $\Phi$. Then, by Proposition~\ref{prop:firsttwo}, Lemma~\ref{lem:degreeone}, and Lemma~\ref{lem:degreegrow}, we have that $\mathcal{A}(\Phi)\frakg_L \subseteq \operatorname{Lie}\{W_1,W_2,W_3\}$. 

It remains to identify $\mathcal{A}(\Phi)\frakg_L$ with~$\xalg(G)$. It is clear that $\mathcal{A}(\Phi)\frakg_L \subseteq \xalg(G)$. We establish below the reverse inclusion. Let $V\in \xalg(G)$. For any $x\in G$, we have that $x^{-1}V(x)\in\frakg$. 
Since $G$ is a subgroup of $\mathrm{SL}_d(\C)$, $x^{-1}=\operatorname{adj}(x)$, the entries of which are thus polynomials.  
Let $X_1,\ldots,X_m$ be a basis of~$\frakg$. Then, there exist
$a_1,\ldots,a_m\in\mathcal{A}(\Phi)$ such that $x^{-1}V(x)=\sum_{i=1}^m a_i(x)X_i$, and hence $V=\sum_{i=1}^m a_iL_{X_i}\in\mathcal{A}(\Phi)\frakg_L$.
\end{proof}

\printbibliography

@article{Andrist,
 author = {Andrist, Rafael B.},
 title = {Integrable Generators of {Lie} Algebras of Vector Fields on {$\mathbb C^n$}},
 journal = {Forum Mathematicum}, volume = {31}, number = {4},
 pages = {943--949}, year = {2019},
 doi = {10.1515/forum-2018-0204}
}

@article{And26,
  author  = {Andrist, Rafael B.},
  title   = {On complete generators of certain Lie algebras on
             {Danielewski} surfaces},
  journal = {International Journal of Mathematics},
  volume  = {37},
  number  = {9},
  pages   = {2650045},
  year    = {2026},
  doi     = {10.1142/S0129167X2650045X}
}

@article{AndristSL,
  author={Andrist, Rafael B.},
  title={Integrable Generators of {Lie} Algebras of Vector Fields on {$\mathrm{SL}_2(\mathbb C)$} and on {$xy=z^2$}},
  journal={The Journal of Geometric Analysis},
  volume={33}, number={8}, pages={240}, year={2023},
  doi={10.1007/s12220-023-01294-x}
}

@misc{BeldievNew,
  author = {Beldiev, Ivan and Pogudin, Gleb},
  title={The {Lie} Algebra of Polynomial Vector Fields on the Affine Space with Constant Divergence is {$1.5$-generated}},
  note = {arXiv:2608.13798}, year = {2026}
}

@article{TV,
 author={T{\'o}th, {\'A}rp{\'a}d and Varolin, Dror},
 title={Holomorphic Diffeomorphisms of Complex Semisimple {Lie} Groups},
 journal={Inventiones Mathematicae}, volume={139}, number={2},
 pages={351--369}, year={2000}, doi={10.1007/s002229900029}
}

@article{TV06,
 author = {T{\'o}th, {\'A}rp{\'a}d and Varolin, Dror},
 title = {Holomorphic Diffeomorphisms of Semisimple Homogeneous Spaces},
 journal = {Compositio Mathematica}, volume = {142}, number = {5},
 pages = {1308--1326}, year = {2006}, doi = {10.1112/S0010437X06002077}
}

@article{Burnside,
 author = {Lomonosov, Victor and Rosenthal, Peter},
 title = {The Simplest Proof of {Burnside's} Theorem on Matrix Algebras},
 journal = {Linear Algebra and its Applications},
 volume = {383},
 pages = {45--47},
 year = {2004},
 doi = {10.1016/j.laa.2003.08.012}
}

@misc{Real,
 author={Jiang, Yiyang and Chen, Xudong},
 title={Integrable Generators of Algebraic Vector Fields for Real Connected Semi-simple Matrix Groups},
 year={2026}, note={arXiv:2608.21670}
}

@book{Humphreys,
 author={Humphreys, James E.},
 title={Introduction to {Lie} Algebras and Representation Theory},
 series={Graduate Texts in Mathematics},
 volume={9},
 publisher={Springer-Verlag},
 address={New York},
 year={1972}
}

@article{Kuranishi,
  author  = {Kuranishi, Masatake},
  title   = {On Everywhere Dense Imbedding of Free Groups in {Lie} Groups},
  journal = {Nagoya Mathematical Journal},
  volume  = {2},
  pages   = {63--71},
  year    = {1951},
  doi     = {10.1017/S0027763000010059}
}

\end{document}